\documentclass{amsart}

\numberwithin{equation}{section}

\usepackage{amsmath, amsthm, amssymb, graphicx, enumerate, datetime, getfiledate, wasysym} 
\usepackage{wrapfig}
\usepackage{xcolor, tcolorbox}
\newcounter{citedtheorems}

\newcounter{theoremcounter}

\newtheorem{defn}[theoremcounter]{Definition}
\newtheorem{theorem}[theoremcounter]{Theorem}

\newtheorem{disc}[theoremcounter]{Discussion}
\newtheorem*{theorem-m}{Theorem \ref{main-theorem}}

\newtheorem*{theorem-x}{Theorem}
\newtheorem*{theorem-abs1}{Theorem \ref{ind-theorem}}
\newtheorem*{theorem-abs2}{Theorem \ref{a23}}
\newtheorem*{theorem-abs3}{Theorem \ref{ind-new}}
\newtheorem*{theorem-abs4}{Theorem \ref{m1}}

\newtheorem{thm-lit}[citedtheorems]{Theorem}
\newtheorem{defn-lit}[citedtheorems]{Definition}
\newtheorem{fact-lit}[citedtheorems]{Fact}
\newtheorem{fact}[theoremcounter]{Fact}

\newtheorem{ntn}[theoremcounter]{Notation}
\newtheorem{cor}[theoremcounter]{Corollary}
\newtheorem{defn-claim}[theoremcounter]{Definition/Claim}

\newtheorem{rmk}[theoremcounter]{Remark}

\newtheorem{expl}[theoremcounter]{Example}

\newtheorem{qst}[theoremcounter]{Question}

\newcommand{\eps}{\epsilon}

\newcommand{\br}{\vspace{2mm}}

\newcommand{\mch}{\mathcal{H}}

\newcommand{\dom}{\operatorname{dom}}

\newcommand{{\xw}}{\mathbf{w}}

\newcommand{\mcx}{\mathcal{X}}

\newcommand{\de}{\mathcal{D}}

\newcommand{\vp}{\varphi}
\newcommand{\Ldim}{\operatorname{Ldim}}
\newcommand{\trv}{\mathbf{t}}
\newcommand{\mca}{\mathcal{A}}

\newcommand{\tlf}{\triangleleft}

\newcommand{\Ex}{\mathop{\mathbb{E}}}

\newcommand{\tcb}[1]{{\color{black}#1}}

\begin{document}

\title{The unstable formula theorem revisited via algorithms}

\author{M. Malliaris}
\address{Department of Mathematics, University of Chicago}
\curraddr{}
\email{mem@math.uchicago.edu}

\author{S. Moran}
\address{Departments of Mathematics, Computer Science, and Data and Decision Sciences, Technion and Google Research}
\curraddr{}
\email{smoran@technion.ac.il}

\begin{abstract} 

This paper is about the surprising interaction of a foundational result from model theory, about stability of theories, with algorithmic stability in learning.  
First, in response to gaps in existing learning models, we introduce a new statistical learning model, called ``Probably Eventually Correct'' or PEC. We characterize Littlestone (stable) classes in terms of this model. 
As a corollary, Littlestone classes have frequent short definitions in a natural statistical sense. In order to obtain a characterization of Littlestone classes in terms of frequent definitions, we build an equivalence theorem highlighting what is common to many existing approximation algorithms, and to the new PEC. This is guided by an analogy to definability of types in model theory, but has its own character. 
Drawing on these theorems and on other recent work, 
we present a complete algorithmic analogue of Shelah's celebrated Unstable Formula Theorem, 
with algorithmic properties taking the place of the infinite. 
\end{abstract}

\maketitle

\vspace{5mm}
\section{Introduction}

\vspace{3mm}

\medskip

The fundamental theorem of statistical learning theory 
says that a set system can be PAC learned if and only if it has finite VC dimension. 
A recent theorem of Alon, Bun, Livni, Malliaris, Moran \cite{ABLMM} proves that a set system can be differentially-privately PAC learned if and only if it 
is a Littlestone class, what model theorists would call stable. 

Although the past decade has led to a deepening and broadening of our understanding of the relation of model theory, specifically stability, and finite combinatorics, see e.g. the introduction to \cite{MiSh:E98}, algorithms are another matter entirely. 
One of the surprises (also to the authors) of \cite{almm} 
was that model theoretic and algorithmic ideas could interact in such an interesting way. 
A motivation for the present paper has been our emerging understanding that certain algorithmic properties of Littlestone classes that arise in learning theory reflect certain infinitary aspects of stability theory which neither combinatorics alone nor model theory seem to capture in the finite. 

\br

The title of the paper refers to Shelah's 1978 ``Unstable Formula Theorem.'' 
For context, across mathematics, we recognize the importance of theorems of equivalences: discoveries that many different phenomena must either all occur together or do not occur at all. Graph quasirandomness of Chung, Graham, Wilson and Thomason is one such 
example. A deep example from model theory, central to our work here, 
is Shelah's Unstable Formula Theorem from \cite{Sh:a}, stated in the Appendix below, which characterizes the stability, 
or equivalently the instability, of a formula in a theory. 

In more detail, certain aspects of Shelah's equivalence are combinatorial, involving existence of half-graphs and trees. 
These combinatorial aspects played a key role in 
\cite{ABLMM}, see \S \ref{s:motiv} below.  
These aspects have also been central to earlier work in model theory and combinatorics, starting with the stable regularity lemma of Malliaris-Shelah 
\cite{MiSh:978}.   But there is much more to the unstable formula theorem, which may seem a priori untranslatable (such as counting types
which 
only detects stability once the sets become infinite, in some sense  dense, see \S \ref{s:revisited}). 
In writing this paper, we have felt it is a   
challenge to our understanding to ask whether the full constellation of structural information it describes remains meaningful in the algorithmic world. 

To this end, we prove three theorems, each of which involves a different kind of work.
First we state and prove a new characterization of Littlestone classes in the language of learning 
by leveraging a model-theoretic sense of what is missing in 
existing models: 
this is 
Theorem \ref{t:stablepec}. As a corollary, we obtain that Littlestone classes have frequent short definitions, \ref{t:discussion}. Next, as motivated at the end of \S \ref{s:frequent1} and the beginning of \S \ref{s:frequent}, we lay out what we believe are the right statistical analogues of definability of types in stable classes in section 
\ref{s:frequent}, leading to Theorem \ref{t:equivalence} on approximations; a certain answer to `what is the 
probabilistic power of stability beyond that of finite VC dimension' is given there. Finally, we arrive to Theorem \ref{t:unstable}, 
the  ``Algorithmic'' Unstable Formula Theorem, which also draws on a wide range of recent work. 

Each of these are motivated and discussed at length in the sections below, 
and written to be accessible to mathematicians in logic, computer science and combinatorics.

%\newpage
\br

\setcounter{tocdepth}{1}

\tableofcontents

\br
\section{Background on Littlestone classes} \label{s:littlestone}

In this paper we will consider concept classes, also called hypotheses classes, or set systems. These are of the form 
$(X, \mch)$ where $X$ is a possibly infinite set and $\mch$ is a set of subsets of $X$. Elements of $\mch$ (or in general, subsets of $X$) 
are sometimes called hypotheses, and may be identified with their characteristic functions. 

Among these, the \emph{Littlestone classes} will play a central role. 
In combinatorial language, 
say that $(X, \mch)$ has Littlestone dimension $d < \omega$ %($d$ is minimal so that) there do not exist
if it is the largest natural number for which there exist
elements $\{ x_\eta : \eta \in {^{\leq d}{\{ 0, 1 \}}} \}$ from $X$ and
%\footnote{In this set-theoretic notation, $d$}
elements $\{ h_\rho : \rho \in {^{d+1}\{0, 1\}} \}$ from $\mch$ 
so that whenever $\eta^\smallfrown \langle 0 \rangle \trianglelefteq \rho$ we have $h_\rho(x_\eta) = 0$, and whenever 
$\eta^\smallfrown \langle 1 \rangle \trianglelefteq \rho$, $h_\rho(x_\eta) = 1$. 
\footnote{We can label the internal nodes of a full binary tree of height $d$ (but not $d+1$) with elements of $X$, and its leaves with elements of $\mch$, so that each $h_\rho$ has the following relation to each $a_\eta$ along its branch: if $\rho$ extends $\eta^\smallfrown \langle 0 \rangle$ then, $a_\eta \notin h_\rho$, and if $\rho$ extends $\eta^\smallfrown \langle 1 \rangle$, then $a_\eta \in h_\rho$.}
In the language of learning, a mistake tree of depth $t$ is a 
complete binary decision tree whose internal nodes are labeled by elements of $X$. 
Each walk forward from a root to a leaf can be described by a sequence 
of $t$ pairs $(x_i, y_i) \in X \times \{ 0, 1 \}$ where~$x_i$ is the label of the node at step~$i$ and $y_i$ is the direction (left or right) in which 
we take the next step. Such a sequence can be seen as a partial characteristic function. 
A mistake tree is said to be shattered by $\mch$ if each such partial function 
is extended by some $h \in \mch$. The Littlestone dimension, $\Ldim(\mch)$, is the (least) upper bound on the depth of a complete binary tree shattered by $\mch$.
Such trees are also called Littlestone trees. 
Littlestone classes are those with finite $\Ldim$.

Littlestone classes were studied by Littlestone \cite{littlestone} and by Ben-David, P\'{a}l, and Shalev-Shwartz \cite{bdps} who proved that 
$\Ldim$ characterizes \emph{online learnability} of the class.\footnote{Here ``online'' refers to the sequential presentation of the elements along the branch that the adversary follows in response to our moves (see for instance \cite{bdss} \S 21).}
It was observed by Chase and Freitag \cite{chase-freitag} that Littlestone dimension 
was the same as model-theoretic 2-rank, and they used this to give new examples of Littlestone classes.  Still, to our knowledge, the community did not place a lot of emphasis on  Littlestone classes, nor was it expected that online learning would have any connection to statistical learning.

The picture is now quite different following \cite{almm}, where it was established that in fact statistical learning and stability/Littlestone are deeply connected, as we now explain.

\section{Background on P.A.C. learning and Differential Privacy} \label{s:pac}

We begin with a quick introduction of the probably approximately correct, or PAC, learning model.
This model aims to capture tasks in which a learning algorithm is given a sequence of training examples
$(x_1,c(x_1)),\ldots, (x_n,c(x_n))$, labeled by some unknown target concept $c:X\to\{0,1\}$ and an (arbitrary) test point $x$,
and its goal is to correctly predict the label $c(x)$. That is, after seeing the training set $(x_i,y_i)_{i=1}^n$
the learner outputs a hypothesis $h:X\to \{0,1\}$ which should be ``\emph{similar}'' to the target concept $c$.
In the PAC model it is assumed that the $x_i$'s and the test point $x$ are sampled i.i.d.\ from a distribution $D$ over $X$ 
which is unknown to the learner,
and that the target concept $c$ belongs to a concept class $\mch$, which is known to the learner (the class $\mch$, not the concept $c$).

We now set up the model in a more precise way, and introduce some basic notation that will be used throughout the paper.
Let $X$ be a set called the domain. 
In what follows, since we will need measurability, we can either assume $X$ is countable, or use standard assumptions which are stated explicitly in Remark~\ref{r:measure} below. 
A hypothesis or concept is an ``$X\to\{0,1\}$'' function.
A concept/hypothesis class~$\mch$ is a set of $X\to\{0,1\}$ functions.
When $D$ is a distribution over $X\times\{0,1\}$ and $h$ a hypothesis,
the loss (sometimes called the population loss) of $h$ with respect to $D$ is defined by 
\begin{equation} \label{e:loss} 
L_D(h)=\Pr_{(x,y)\sim D}[h(x)\neq y]. 
\end{equation}
The distribution $D$ is said to be realizable by $\mch$ if 
\begin{equation} 
\inf_{h\in\mch }L_D(h)=0.
\end{equation}
For a set $Z$, let $Z^\star = \cup_{n=0}^\infty Z^n$. A learning rule $\mca$ is a mapping that takes as an input a finite sequence of examples and outputs a hypothesis. That is, it is function $\mca:(X\times\{0,1\})^\star \to \{0,1\}^{X}$.
Given a distribution $D$ over $X\times\{0,1\}$, the \emph{learning curve} of $\mca$ with respect to $D$ is the sequence
\[ n\mapsto \Ex_{S\sim D^n}[L_D(h_n)],\]
where $h_n= A(S)$ is the hypothesis outputted by the algorithm on the input sample~$S$.
That is, the learning curve measures the expected performance of the algorithm as a function of the number of examples it is trained on.

$\mch$ is said to be PAC learnable if there is a vanishing sequence $\alpha(n)\to_{n\to\infty} 0$ and a learning rule $\mca$
such that for every realizable distribution $D$:
\begin{equation} 
\Ex_{S\sim D^n}[L_D(h_n)]\leq \alpha(n),
\end{equation}
where $S$ is the input sample $S=\{(x_i,y_i)\}_{i=1}^n$ and $h_n= \mca(S)$ is the output hypothesis of the learning rule $\mca$.
In other words, all realizable learning curves of $\mca$ converge uniformly to $0$.\footnote{In Section~\ref{s:pec} we discuss variants of this definition where we relax the uniform requirement.}

A couple of remarks for readers who may have seen other (equivalent) definitions of PAC learnability: 
the above definition asserts that the algorithm $\mca$ ``probably'', ``approximately learns'' $\mch$
in the sense that in expectation it achieves a small loss given enough examples.
An equivalent way of formulating it requires a high probability bound on the event of achieving a small loss 
(rather than taking expectation as above). In the second formulation there is an additional parameter $\beta$
quantifying this confidence. We chose the first formulation because it uses the learning curve perspective which will be useful in later sections.

To summarize informally, a concept class is PAC-learnable if there exists a learning algorithm which with high probability (``probably'') outputs a concept which is close to the correct one (``approximately correct''). 
The fundamental theorem of statistical learning theory says that a set system is PAC learnable if and only if it has 
finite VC dimension.  See for instance \cite{bdss} \S 6.4-6.5. 

\medskip

We now arrive to a central idea for the paper.

Differentially-private PAC learning asks, informally, for PAC learnability 
via an algorithm which is robust in the sense that small changes in the input do not disturb the output.   There are 
obvious practical uses of such an algorithm in being able to guarantee privacy or anonymity to individual 
participants in data sets. It was introduced by  
Dwork, McSherry, Nissim and Smith in \cite{dmns}, for which they won the G\"odel prize. 

Formally stating the definition of privacy requires several steps. 
In generality, every algorithm induces a mapping taking inputs to outputs.
When the algorithm is randomized, we can think about it as a (deterministic)
mapping from inputs to distributions\footnote{In more detail, a randomized algorithm is modeled as a function $f:I\times R\to O$, where $I$ is the input space of the algorithm, $O$ is the output space of the algorithm, and $R$ is the set of random seeds. This means that there is a distribution over $R$ called the ``randomness'' of the algorithm, and given a point $i$ the output of the algorithm is computed by sampling $r$ from $R$ and outputting $f(i,r)\in O$. Thus, every such algorithm induces a deterministic mapping from $I$ to distributions over $O$, by assigning to each $i\in I$ the push-forward measure of the randomness on $R$ induced by~$f(i,\cdot)$.} over outputs 
(the output is fully determined once the internal randomness of the algorithm is instantiated).
Thus, below we interpret a randomized learning algorithm as a deterministic mapping
from input samples to distributions over hypotheses.  
This interpretation
is used to define privacy. 
Say that two probability distributions $p, q$ are indistinguishable, so-called $(\epsilon, \delta)$-indistinguishable, when for  
every event $E$, $p(E) \approx_{\epsilon, \delta} q(E)$ where $a \approx_{\epsilon, \delta} b$ means that 
$a \leq e^\epsilon b + \delta$ and $b \leq e^\epsilon a + \delta$. 
Note that for small $\epsilon,\delta$, the affine map $x\mapsto e^\epsilon x + \delta$ is close to the identity map~$x\mapsto x$. Thus,  $a\approx_{\eps,\delta} b$ roughly means that there is an approximate identity which maps $a$ to $b$ and vice versa.
For more on this notion of similarity, see Vadhan \cite{vadhan} 7.1.6.

Given $\epsilon$ and $\delta$, we say that a learning algorithm $\mathcal{A}$ is $(n,\eps,\delta)$-differentially private if whenever $\mca$ is presented with two sequences of the domain of a given length~$n$ which differ in one place, then the distributions it outputs in each case are $\approx_{\epsilon,\delta}$ indistinguishable. Here $\eps$ is a constant (independent of input-sample size) whereas $\delta$ should be negligible\footnote{I.e.\ if $n$ denotes the sample size then $\lim_{n\to\infty} \delta(n)\cdot P(n)=0$ for every polynomial $P$.} compared to the input size. 
So-called pure differential privacy is the special case where $\delta = 0$. We focus on the more general case, approximate differential privacy (or just: differential privacy) which allows $\delta$ to be nonzero. 

More precisely, a concept class is said to be differentially-privately PAC-learnable if there exists a learning rule $\mca$, a vanishing sequence $\alpha(n)\to 0$, and a negligible sequence $\delta(n)\to 0$, such that $\mca$ is $(n,\eps,\delta(n))$-differentially private with $\eps\leq 1$, and also $\mca$ PAC learns $\mch$: namely for every realizable distribution $D$ and every $n$, $\Ex[L_D(h_n)\leq \alpha(n)$, where $h_n=A(S)$, $S$ is the input sample, and the expectation is taken over $S\sim D^n$ as well as the randomness of $\mca$.\footnote{Notice that we bound the privacy parameter $\eps$ by $1$ in this definition. This choice is arbitrary: replacing $\eps$ with any other positive real number yields an equivalent definition, see~\cite{vadhan} for more details.}
Observe that there are four parameters implicit in this definition: the sample complexity $n$, the accuracy parameter $\alpha$, and the privacy parameters $\epsilon, \delta$ related to $\approx_{\epsilon, \delta}$.

\vspace{5mm}

\section{Motivation for our new learning model}\label{s:motiv}

At this point let us state the characterization theorem more precisely following~\cite{ABLMM}.  The first direction is due to 
Alon, Livni, Malliaris, Moran in \cite{almm} and the second is due to Bun, Livni, Moran in \cite{BLM}.  There have since been 
quantitative improvements to the bounds in the second direction~\cite{GGKM21}.

%\newpage

\begin{thm-lit} Private learning and finite Littlestone dimension are equivalent: 
\begin{enumerate}
\item \emph{(Any privately learnable class must have finite Littlestone dimension.)} 
\\
%\textcolor{red}
{Suppose there that for some $n\in\mathbb{N}$, $\eps>0, \alpha < \frac{1}{2}$, and $\beta<1$ there is a differentially private learning rule for $\mch$ with sample complexity $n$, privacy parameters $\eps$, $\delta(n)=O\bigl(\frac{1}{n^2\log n}\bigr)$, and accuracy and confidence parameters $\alpha$ and $\beta$.
Then $\mch$ is a Littlestone class, i.e.\ its
$\Ldim$ %Littlestone dimension
is finite.}

\item \emph{(Any class with finite Littlestone dimension is privately learnable)} 
\\ Suppose $\mch$ has Littlestone dimension $d$, 
let $\epsilon, \delta \in (0, 1)$ be privacy parameters, and let $\alpha, \beta \in (0, \frac{1}{2})$ be accuracy parameters. For 
$n = O_d(\frac{\log(1/\beta \delta)}{\alpha \epsilon})$, there exists an $(\epsilon, \delta)$-differentially private algorithm such that for 
every realizable distribution $D$ and input sample $S \sim D^n$, the hypothesis $f = \mathcal{A}(S)$ satisfies 
$\operatorname{loss}_D(f) \leq \alpha$ with probability at least $1-\beta$, where the probability is taken over $S \sim D^n$ as well as the internal randomness of $\mca$. 
\end{enumerate}
\end{thm-lit}

Several points about the existing proofs of this theorem may reasonably be understood as challenges to our understanding. 
The proof of (2) proceeds by proving something of independent interest: that any class with finite Littlestone dimension 
can be learned by a so-called \emph{globally stable} algorithm, and indeed this is characteristic of Littlestone classes, 
where: 

\begin{defn}[\cite{BLM}] \label{d:frequent} 
An algorithm $\mca$ is called $(n, \eta)$-globally stable with respect to a distribution $D$ if there exists a hypothesis $h$ such that 
\[ \operatorname{Pr}_{S \sim D^n} [A(S) = h] \geq \eta \]
i.e. hypothesis $h$ is outputted with frequency at least $\eta$. 
\end{defn}

\begin{thm-lit}[\cite{BLM}] \label{t:global}
A class $(X, \mch)$ is Littlestone if and only if it is learnable by a globally-stable algorithm %\textcolor{red}
{in the following sense:
there exists an $\eta >0$ such that for every $\alpha,\beta > 0$ there is an $n$ and an $(n,\eta)$-globally stable rule
whose accuracy and confidence parameters are $\alpha,\beta$.}
\end{thm-lit}

Note that the hypothesis of PAC-learnable alone does not guarantee any repetition in output hypotheses: the output must often be close to the target concept, but a priori, it could be different every time. 

Moreover, the frequent hypothesis $h$ of Definition \ref{d:frequent} is still subject to the error allowance of PAC learning: there is  
no guarantee that $h$ is actually correct, only that the loss is less than our nonzero $\alpha$ given in advance.  Although 
$\eta$ in Definition~\ref{d:frequent} is nonzero, it is typically small, for instance in \cite{ABLMM} $n = 2^{O(d)}$ and 
$\eta = 2^{-2^{O(d)}}$. 

Given all we know about Littlestone (stable) classes, this begs the question of whether we may improve along \emph{both} these axes: 

\begin{qst}
Can Littlestone classes be learned by a learner which eventually almost always outputs the same frequent hypothesis (i.e.\ the global stability parameter is close to $1$)? Notice that if this happens then the frequent hypothesis must have error $0$, because it is frequently outputted for arbitrarily large sample sizes.
\end{qst}

The next section answers these questions and to do so introduces a new learning model which we call Probably Eventually Correct learning. 
Below we will start with a more precise introduction and then point out where the subtleties are in the definitions and proofs. 

%\newpage 

\vspace{5mm}

%\section{Probably eventually correct = PEC learning} 
\section{A new statistical learning model: Probably Eventually Correct}
\label{s:pec}

In this section we define and prove a characterization of Littlestone classes via a new statistical learning model which we call 
PEC or probably eventually correct learning. In light of the above discussion,  
the gap filled by this new theorem can be motivated in several ways.

\begin{itemize}
\item Previous work on global stability had characterized Littlestone classes as being PAC-learnable by an algorithm which outputs some fixed hypothesis, close to the target concept, with nonzero probability.  
However the frequency of this hypothesis, though nonzero, was typically small.

\item Moreover, its loss, though small, was a priori nonzero.

\item Although Littlestone classes carry a natural ``canonical algorithm,'' the so-called standard optimal algorithm or SOA (\S \ref{d:soa})
and although various known algorithmic characterizations of Littlestone classes in the statistical setting 
(such as differentially private learning, globally stable learning, and other characterizations, some of which are presented in the next section) 
use the SOA as a key sub routine in the corresponding algorithms, the SOA by itself does not satisfy the necessary definitions (e.g.\ the SOA is not differentially private, but it is used as a sub routine in differentially private algorithms for Littlestone classes). 
\end{itemize}

The PEC model we propose requires that almost surely the algorithm eventually outputs hypotheses which are correct on the distribution. 
(Hence the name \emph{Probably Eventually Correct}.)
Informally, imagine we sequentially feed a learning algorithm with training examples drawn from the population,
and let $k$ be the number of times such an algorithm 
revises or changes its mind before stabilizing (see Definition \ref{d:mindchanges} below). In the section's main theorem, \ref{t:stablepec}, we prove that a class is Littlestone with dimension $d$ if and only if it can be PEC-learned by an algorithm with $k \leq d$. We call this ``stable PEC learning'', because the algorithm stabilizes to an absolutely correct (up to measure zero) hypothesis after at most $k$ mind changes. Here $k$ does not depend on the distribution although, unlike global stability, the sample size is not bounded, that is, the last mind change can occur potentially after seeing very many examples. 
We will see below that any Littlestone class is naturally stably  PEC-learned by its ``canonical'' learning rule, namely the SOA;
the converse direction, that any stably PEC learnable class is Littlestone, is less obvious; this is also reflected by its proof which is more subtle.

Theorem \ref{t:stablepec} responds to the three points above as follows.  
Let $\mca$ be a stable PEC learner for a 
Littlestone class $\mch$ with $\Ldim(\mch) = d$. 
We shall evaluate the performance of $\mca$ by 
choosing a countably infinite sequence $\{ (x_n,y_n) \}$ sampled i.i.d.\ from $X\times\{0,1\}$ according to an arbitrary but 
fixed realizable distribution $D$,  
writing $h_n$ for the output of $\mca$ 
on $\{ (x_1,y_1), \dots, (x_n,y_n) \}$.  We will see that:  

\begin{itemize}
\item The frequency of the output hypothesis is $1$ (up to measure zero). \footnote{More precisely, 
for any realizable distribution $D$ over $X$, there is $h = h(D)$ such that with probability $1$, 
given any countably infinite sequence sampled i.i.d.\ from $X$ according to $D$, the $h_n$'s are eventually equal to $h$, up to measure zero, meaning $\Pr[h(x)\neq h_n(x)]=0$.} 

\item The loss of the output hypothesis is eventually zero. (See \S \ref{s:pac}.) %Definition~\ref{def:loss}.) 

\item The SOA is a stable PEC learner for $\mch$, i.e.\ it stabilizes on the eventual hypothesis after at most $d$ mind changes.
\end{itemize}

There are other advantages which we will discuss later in the presence of other models, 
such as what might be called the 
frequent definitions phenomenon. 

To begin we propose two definitions. 

\begin{defn} \label{d:pec-learnable} We say that $\mch\subseteq\{0,1\}^\mcx$ is \emph{Probably Eventually Correct (PEC) Learnable}
	if there exists a learning rule $\mca$ such that for every distribution $D$ which is realizable by $\mch$
	(i.e.\ $\inf_{h\in \mch}L_D(h)=0$) it holds that
	\begin{equation}\label{e:pec}
	\Pr\Bigl[(\exists N)(\forall n\geq N):L_D(h_n)=0\Bigr] = 1,
	\end{equation}
	where the probability is taken over sampling an infinite IID sequence of examples $S=\{z_n\}$ from $D$,
	and $h_n$ inside the probability is the hypothesis outputted by $\mca$ on the prefix $S_n$ of $S$ 
	which consists of the first $n$ examples in~$S$.
	\end{defn}
In words, Equation~\ref{e:pec} says that almost surely $\mca$ outputs hypotheses which are eventually correct on $D$.
	
\begin{defn} \label{d:mindchanges} We further say that $\mca$ PEC learns $\mch$ in a \emph{stable} fashion 
	if there exist $d\in\mathbb{N}$ such that 
	in addition to Equation~\ref{e:pec} the following holds:
	\begin{equation}\label{e:stable}
	\Pr\Bigl[\Bigl\lvert\{ n : h_n \neq h_{n+1}\}\Bigr\rvert \leq d\Bigr] = 1.
	\end{equation}
	That is, $\mca$ changes its output hypothesis at most $d$ times while processing the infinite sequence.
	(We stress that $d$ does not depend on the distribution $D$; 
	moreover, it can be shown that allowing $d$ to depend on the distribution boils down to PEC learnability\footnote{More precisely, any PEC learnable class $\mathcal{H}$ admits a learning rule $\mca$ satisfying that for every realizable distribution $D$ there exists an integer $d=d(D)$ for which Equation~\ref{e:stable} holds. }.)
	\end{defn}

\subsection{The SOA} \label{d:soa}
Recall that the SOA, or standard optimal algorithm, 
predicts by optimizing the $\Ldim$ if possible. That is, given any labeled sample 
$S = \{ (x_i, y_i) : i = 1, \dots, n \} \subseteq X \times \{ 0, 1 \}$ 
let $\mch_S = \{ h \in \mch: S \subseteq h \}$. If the SOA receives a sample $S$, first for each $x \in X$ it chooses 
$y_x \in \{ 0, 1 \}$ so that $\Ldim (\mch_{S \cup \{ (x, y_x) \}})$ is maximized, and in case of ties 
(or if the $\Ldim$ is 
undefined), chooses $y_x = 1$. Then it 
outputs $h = \{ (x, y_x) : x \in X \}$.   See e.g.\ \cite{bdss}.  As stated, the SOA makes sense for any class but it is only really 
useful when the Littlestone dimension is defined (otherwise the tiebreaking clause might be invoked everywhere). 

The characterization of PEC learning below mentions a variant of the SOA 
    which is defined for classes $\mch$ which are non Littlestone (i.e.\ which shatter arbitrarily large finite trees).
    This was defined and studied in~\cite{universalearning2020}, see below.
    The idea is to extend the definition of Littlestone dimension from natural numbers to ordinal numbers.
    It turns out that a class $\mch$ has a well defined ordinal Littlestone dimension
    if and only if $\mch$ does not shatter a complete infinite binary tree. E.g.\ the class of all thresholds over $\omega$
    has Littlestone dimension $\omega$ (more generally, the class of thresholds over an ordinal $\alpha$ has Littlestone dimension~$\alpha$).  
    The crucial property of the extended definition is that for every ordinal Littlestone class $\mch$
    and for every $x\in X$, %either
    at least one of 
    $\{h\in\mch : h(x)=0\}$ or $\{h\in\mch : h(x)=1\}$ have strictly smaller (ordinal)
    Littlestone dimension than $\mch$. Thus, the \emph{ordinal SOA} follows precisely the same strategy as the SOA
    of predicting via optimizing the ordinal Littlestone dimension.
    (Notice that this strategy guarantees that the ordinal SOA makes a finite (but perhaps unbounded) number of mistakes
    on every realizable sequence.)
\begin{ntn}\label{not:lazy}
The SOA algorithm has the property that it is {\it lazy} in the sense that it does not change its output hypothesis unless it makes a mistake, 
meaning that for all $n$, if $h_n(x_{n+1})=y_{n+1}$ then $h_{n+1}=h_n$.
\end{ntn}

\subsection{Prior work: universal learning} 
We aim to work in the statistical (say, as opposed to online) setting, and will build on the recent paper  \cite{universalearning2020} which 
deals with certain limitations of the PAC model compared to what is  called \emph{universal learning}.  
In the PAC model the performance of an algorithm $\mca$ is 
measured by looking at its PAC learning curve, that is, the function $\varepsilon$ which on each natural number $n$ returns the worst-case error $\mca$ may be expected to make on $n$ i.i.d.\ training examples, sampled from \emph{any} (in particular, from a ``worst-case'') realizable distribution. 
A priori, the distribution witnessing~$\varepsilon(n)$ may be different from the distribution witnessing $\varepsilon(n+1)$.  
\cite{universalearning2020} makes the case for studying a different measure of performance. Fixing a distribution $D$, look at the $D$-learning curve for $\mca$ (this is the function $f_D(n)$ which returns the error $\mca$ is expected to have when getting an input sample of $n$ i.i.d. training examples sampled from $D$).  
So we say that $\mch$ is \emph{universally} learned at e.g.\ an exponential rate if there is an algorithm $\mca$ so that for every realizable distribution $D$,
the $D$-learning curve is upper bounded by $C\cdot\exp(-c\cdot n)$, where $C,c$ are constants that might depend on $D$ (but not on $n$).
The main result of \cite{universalearning2020} is a trichotomy theorem characterizing the possible optimal universal learning rates for hypothesis classes. 
%WAS: the optimal universal learning rate of every class $\mch$.
Along the way, they establish the following useful fact: 

\begin{thm-lit}[\cite{universalearning2020} \S 4.1 -- \S 4.3] \label{universal-fact}
$\mch$ does not shatter an infinite Littlestone tree if and only if there exists a learner $\mca$ such that for every realizable distribution $D$,
\begin{equation} \label{eq:prev}
\lim_{N\to\infty} \Pr\Bigl[L_D(h_N)=0\Bigr] = 1.
\end{equation} 
\end{thm-lit}

Theorem \ref{universal-fact} does not fit our purposes out of the box\footnote{In particular, notice that the term ``eventually correct'' is informally used in~\cite{universalearning2020} to describe Theorem~\ref{universal-fact}, but it is not hard to see that Equation~\ref{eq:prev} is not equivalent to the definition of PEC learning in the sense that there exist learning rules that satisfy it but are not PEC learners.}, but inspecting its proof a very useful feature appears: 

\begin{fact} \label{fact-soa}
If $($and only if$)$ the equivalent conditions of Theorem \ref{universal-fact} are satisfied, then they are satisfied by taking $\mca$ to be the ordinal SOA. 
\end{fact}

To recap some points from \S \ref{d:soa}, since we have defined the SOA to make sense for any class, the reader can take ``ordinal SOA'' in Fact \ref{fact-soa} 
to mean ``SOA''. 
Recall from the discussion there that ``$\mch$ does not shatter an infinite Littlestone tree'' does not necessarily imply ``$\mch$ does not shatter some Littlestone tree of finite height.'' (Model theorists should understand this as a 
feature of the absence of compactness.)  
The presence of an ordinal SOA is good news and bad news for our present purpose (of finding a natural and simple statistical characterization of Littlestone classes).
Indeed, it follows from the characterization in~\cite{universalearning2020} that there is no finer universal characterization of Littlestone classes within the ordinal ones.

So identifying the role of finite $\Ldim$, our Littlestone, 
requires something else.  
The key contribution of the section's first theorem, Theorem \ref{t:pec}, 
is in identifying Definition~$\ref{d:pec-learnable}$ as 
a convenient and natural framework for our purposes.
Once this definition is in place, stating and proving \ref{t:pec} 
as we do below via Theorem \ref{universal-fact} is not difficult. 
Theorem \ref{t:pec} sets the stage for a characterization of finite Littlestone via Theorem \ref{t:stablepec}.  This second theorem requires a proof (specifically $(1)\implies (2)$ as discussed there).
We have called the learning model stable PEC learning since, as discussed, it amounts to a kind of stability of the algorithm.

In the rest of this section, we will prove the two promised theorems.  

\begin{theorem}[PEC Learning]\label{t:pec}
The following are equivalent for a class~$\mch$:
\begin{enumerate}
\item $\mch$ is PEC learnable,
\item $\mch$ does not shatter an infinite Littlestone tree.
\end{enumerate}
Moreover, whenever $\mch$ satisfies the above items, then 
the $($ordinal$)$ SOA is a PEC learner for $\mch$.
\end{theorem}

\begin{theorem}[Stable PEC Learning]\label{t:stablepec}
The following are equivalent for a class $\mch$:
\begin{enumerate}
\item $\mch$ is PEC learnable in a stable fashion,
\item $\mch$ does not shatter arbitrarily large finite Littlestone trees. 
(Equivalently, $\mch$ is a Littlestone class.)
\end{enumerate}
Moreover, whenever $\mch$ satisfies the above items, then the SOA is a PEC learner for~$\mch$ that makes at most $d=\Ldim(\mch)$ many mind changes (i.e.\ changes its ouptut hypothesis at most $d$ times on every countably infinite realizable sequence).
\end{theorem}

\subsection{Proofs}

\begin{rmk} \label{r:measure}
The PEC learning model is defined using the language of probability theory. As such, it implicitly assumes a measure space  (and in particular a $\sigma$-algebra) which defines the set of distributions that are used. In fact, \cite{universalearning2020} prove Theorem~\ref{universal-fact} under a standard measurability assumption from empirical process theory, where it is usually called the image admissible Suslin property.\footnote{Specifically that the domain $X$ is a Polish space (a topological space with a dense countable set that can be metrized by a complete metric), and that the hypothesis class $\mch$ admits a  measurable parametrization in the sense that there is a Polish parameter space $\Theta$ and a Borel-measurable map $\mathtt{h}:\Theta\times X\to\{0,1\}$ such that $\mch=\{\mathtt{h}(\theta,\cdot) : \theta\in\Theta\}$ (see Definition 3.3 in~\cite{universalearning2020}).}. 
Thus, we inherit this assumption from  \cite{universalearning2020} (but in fact this assumption is also implicit in uniform convergence from PAC learning).
However, the reader can simply have in mind the stronger assumption that the domain $X$ is countable and that the $\sigma$-algebra is the entire power set of $X$.  
\end{rmk}

\begin{proof}[Proof of Theorem~\ref{t:pec}]

Let $D$ be a distribution realizable by $\mch$ and let $S=\{z_n\}_{n=1}^\infty$ be i.i.d examples $z_n=(x_n,y_n)$ drawn from $D$. 
Recall that a learning rule $\mca$ PEC learns $\mch$ if
\begin{equation} \label{e:one} 
\Pr\Bigl[(\exists N)(\forall n\geq N): L_D(h_n)=0 \Bigr] = 1, 
\end{equation} 
where $h_n$ is the hypothesis outputted by $\mca$ when trained on the first $n$ examples from $S$.
By $\sigma$-additivity (\ref{e:one}) is equivalent to 
\begin{equation}\label{eq:needtoshow}
\lim_{N\to\infty} \Pr\Bigl[(\forall n\geq N): L_D(h_n)=0\Bigr] = 1.    
\end{equation}
Thus, we need to show that $\mch$ does not shatter an infinite Littlestone tree if and only if there exists a learning rule $\mca$ satisfying Equation (\ref{eq:needtoshow}) for every realizable distribution $D$.

Since (\ref{eq:needtoshow}) implies (\ref{eq:prev}), it implies (by Theorem \ref{universal-fact}) {that $\mch$ 
does not shatter an infinite Littlestone tree. }
This shows $(1)\implies (2)$.

For $(2) \implies (1)$, assume $\mch$ does not shatter an infinite Littlestone tree. Then by Theorem \ref{universal-fact} 
there exists a learning rule $\mca$ satisfying Equation (\ref{eq:prev}).
Moreover, {by Fact \ref{fact-soa}, we may take $\mca$ to be the $($ordinal$)$ SOA}. 
By the lazy quality of the SOA, see~\ref{not:lazy}, for all $N$
\begin{align*}
    &\Pr \Bigl[(\forall n\geq N): L_D(h_n)=0 ~\big\vert~ L_D(h_N) = 0\Bigr]\geq\\
     &\Pr\bigl[(\forall n\geq N): h_N(x_n)= y_n ~\big\vert~ L_D(h_N)=0\bigr]=1,
\end{align*}
where the above inequality ``$\geq$'' follows because the ordinal SOA is lazy.
Thus, by the definition of conditional probability it follows that the ordinal SOA satisfies that for all $N$, 
\[ \Pr\Bigl[L_D(h_N)=0\Bigr] = \Pr\Bigl[(\forall n\geq N): L_D(h_n)=0\Bigr].\]
In particular, since the ordinal SOA satisfies Equation (\ref{eq:prev}), {the limit as $N \rightarrow \infty$ of the left-hand side is $1$, 
so the same is true of the right-hand side, so (\ref{eq:needtoshow}) holds} and hence it PEC learns $\mch$.
\end{proof}

\begin{proof}[Proof of Theorem~\ref{t:stablepec}]
Consider first the direction $(2) \implies (1)$.
Let $\mch$ be a class satisfying Item $(2)$. Thus, $\mch$ is a Littlestone class, and by Theorem~\ref{t:pec} the SOA is a PEC learner for $\mch$. 
Thus, the SOA makes at most $d=\Ldim(\mch)$ many mistakes on every realizable sequence, and therefore also on (almost) every random i.i.d sequence~$S$ which is drawn from a {given} realizable distribution $D$. The latter, combined with the fact that the SOA is lazy (see \ref{not:lazy}) implies that it makes at most $d$ mind changes on (almost\footnote{i.e.\ with probability $1$.}) every sequence $S$ which is drawn from a realizable distribution $D$.

We now turn to prove the more challenging direction $(1)\implies (2)$. 
    Our strategy is to consider the contra-positive statement and show that if $\Ldim(\mch) > d$
    then every learner that makes at most $d$ mind changes is not a PEC learner for $\mch$.
    In more detail, let $\mca$ be a learning rule that makes at most $d$ mind changes 
    on almost every sequence $S$ which is drawn from a distribution $D$ that is realizable by $\mch$.
    Assumes towards contradiction that $\mca$ PEC learns $\mch$.

It is enough to prove that there exists a finite realizable sequence $R_\ell=\{z_i\}_{i=1}^\ell$ 
    such that $\mca$ makes $d$ mind changes on the subsequence $R_{\ell-1}=\{z_i\}_{i=1}^{\ell-1}$ 
    and $h(x_\ell)\neq y_\ell$, where $h$ is the output hypothesis of $\mca$ on $R_{\ell-1}$ and $z_\ell=(x_\ell,y_\ell)$. 
    Indeed, if this holds then pick $D$ to be uniform over the examples in $R_\ell$, 
    and notice that with a positive probability over drawing a countably infinite sequence $S$ from $D$, 
    it holds that $R_\ell$ is a prefix of $S$.
    On that event, $\mca$ outputs $h$ infinitely often (on all prefixes of $S$ that contain $R_\ell$), 
    and since \tcb{$L_D(h)\geq \ell^{-1}>0$} it follows that $\mca$ does not PEC learn $\mch$,
    which is a contradiction.

It thus remains to construct the sequence $R_\ell$.
Let $T$ be a tree of depth $d+1$ which is shattered by $\mch$. 
We prove by induction that the sequence $R_\ell$ can be constructed so that its examples belong to a branch of $T$.

Consider the base case $d=0$; thus $T$ has depth $1$. Let $x\in \mcx$ denote the instance labeling the root of $T$. Since $\mca$ makes $0$ mind changes, then it outputs an hypothesis $h_0$ which does not depend on the input sample. 
In particular the sequence $R_1=\{(x,1-h_0(x))\}$ which consists of a single example satisfies the requirement.

For the induction step, assume $d>0$, and let $T$ be a tree of depth $d+1$. 
So, we want to prove that there exists a sequence $R_\ell$ whose examples belong to one of the branches of $T$ such that $\mca$ makes $d$ mind changes on $R_{\ell-1}$ and $h_{\ell-1}(x_\ell)\neq y_\ell$.
Let $T_{d}$ be the subtree of $T$ obtained by removing the last level. 
By induction hypothesis, there exists a sequence $R_m=\{z_i\}_{i=1}^{m}$ whose examples belong to a branch of $T_d$ such that $\mca$ makes $d-1$ mind changes on $R_{m-1}$, and $h_{m-1}(x_m)\neq y_m$, where $h_{m-1}$ is the hypothesis outputted by $\mca$ on $R_{m-1}$. 
Let $D_m$ denote the uniform distribution over the examples in $R_{m}$.  
Let $S$ be an infinite i.i.d sequence drawn from $D_m$
and consider the event $E$ that $R_m$ is a prefix of $S$ (in the same order). 
Note that $E$ has a positive measure.
Now condition on $E$ and imagine simulating the algorithm $\mca$ on $S$. 
After processing $R_{m-1}$, we have that $\mca$ made $d-1$ mind changes.
Next, it will process $(x_m,y_m)$ and by assumption, the hypothesis $h_{m-1}$
misclassifies $x_m$; now this does not mean that $\mca$ will change its mind at this point,
but eventually, since $\mca$ is assumed to PEC learn $\mch$ it must change its mind.
Indeed, otherwise its loss on $D_m$ will be at least $m^{-1}>0$. 

This shows that indeed there exists a finite sequence~$R_{\ell-1}$ such that
(i) $R_m$ is a prefix of $R_{\ell-1}$,
(ii) every example in $R_{\ell-1}$ already appears in $R_m$,
(iii) $\mca$ makes exactly $d$ mind changes on $R_{\ell-1}$.
Let~$h_{\ell-1}$ denote the hypothesis outputted by $\mca$ on $R_{\ell-1}$
and let $x_\ell$ be the child of $x_m$ corresponding to the label~$h_{\ell-1}(x_m)$.
The sequence~$R_\ell$ obtained by concatenating $\{(x_\ell, 1 - h_{\ell-1}(x_\ell))\}$ 
satisfies the requirement. This finishes the proof.
\end{proof}

\section{Littlestone classes have frequent definitions} 
\label{s:frequent1}
 
Before continuing, let us derive from the PEC theorem 
\ref{t:stablepec}  
what we might call a 
	{\it frequent definitions phenomenon}: namely that stable concepts
	have \emph{short} definitions $($they are defined using at most $d$ examples, where $d$ is the Littlestone dimension$)$,
	and that these definitions are \emph{frequent} in the sense that a typical sequence
	of IID examples contains such a definition almost surely.

\begin{cor} \label{t:discussion}
If $\mch$ is a Littlestone class of dimension $d$, then:

\begin{enumerate}

\item[(A)] Every hypothesis $h \in \mch$ has a definition of size $\leq d$, meaning that there is $h^\prime \subseteq h$ which is a partial function with 
$|\dom(h^\prime)| \leq d$ such that $h = \operatorname{SOA}(h^\prime)$  [the notation means $\operatorname{SOA}$ reconstructs $h$ from $h^\prime$].

\item[(B)] Every hypothesis $h \in \mch$ has \emph{frequent} definitions of size $\leq d$, meaning that  for any distribution $\de$ over $X$ there is some 
hypothesis $g$ such that:
\begin{itemize}
    \item $g$ differs from $h$ on a set of measure zero, and 
    \item a random sequence of i.i.d.\ examples almost surely contains a definition for $g$ of size $\leq d$, meaning a set from which the SOA reconstructs~$g$. 
\end{itemize} 
\end{enumerate}

\end{cor}

\begin{proof}
(A) follows from the definition of Littlestone dimension. 

(B) This is exactly what we have proved about PEC. 
\end{proof}

Notice that this theorem is ``if then'' and not stated as a characterization. 

As suggested in the theorem, a reasonable interpretation of Item A is that Littlestone classes admit short definitions, and a reasonable interpretation of Item B is that these definitions are abundant (or frequent). This raises the question of whether there is a general meaningful formal definition of ``short definitions'' and ``frequent definitions'' in a way that is satisfied by Item A and B.
Perhaps the most natural way to do it is to define that a class $\mch$ admits \emph{short definitions} of size $\leq d$ if there exists a learning rule $\mca$ such that for every hypothesis $h\in\mch$ there exists a partial function $h'\subseteq h$ of size at most $d$ such that $\mca(h')=h$, and to define that $\mch$ admits \emph{frequent definitions} if in addition for any distribution $\de$ over $X$ there is some hypothesis $g$ such that:
$g$ differs from $h$ on a set of measure zero, and 
a random sequence of i.i.d. examples almost surely contains a definition for $g$ of size $\leq d$, meaning a set from which the learning rule $\mca$ reconstructs~$g$. 

However, this naive definition does not provide a characterization of Littlestone classes. In other words, there are non Littlestone classes $\mch$ which have both short and frequent definitions in the above sense. A simple example is the class $\mch = \{1[x\geq n] : n\in\mathbb{N}\}\subseteq \{0,1\}^\mathbb{N}$. Namely the class of all thresholds over the natural numbers. Clearly this is not a Littlestone class but any threshold is naturally definable by the two points in which the label changes from $0$ to $1$, and for any distribution over $\mathbb{N}$, a random infinite i.i.d sequence will almost surely contain the two points in support of the distribution where the label changes from $0$ to $1$. 

Is there a correct point of view from which we can see that Littlestone classes are indeed characterized by having definitions of certain kinds?  The next section takes up this question.

\vspace{5mm}

\section{Littlestone classes characterized by frequent approximations} \label{s:frequent} 

\setcounter{theoremcounter}{0}

One of the characteristic properties of types in stable theories is that they can be faithfully reconstructed from a very small amount of 
information, so-called \emph{definability of types}. 
This key idea has not found a straightforward translation to learning theory and hypotheses in Littlestone classes. For instance, there is a substantial body of work around \emph{sample compression} in learning theory, leading to many interesting results in VC classes but not necessarily Littlestone ones.
Per the previous section, any hypothesis in a Littlestone class may be exactly reconstructed by applying the SOA to a specific finite segment in the manner of a definition; a priori the probability of drawing the parameters needed by sampling might be zero. PEC significantly changes the picture with its frequent short definitions, but is of course guaranteed to be correct only up to measure zero. After PEC, we may still ask about the role of the SOA and about whether to relax measure zero.

In writing this paper we arrived at the conclusion that the correct analogue of ``definition'' for this context is hidden in plain sight.  

We propose the heuristic that \emph{learning is the existence of approximations}, that is, learning algorithms are themselves what we might call approximate definitions for target concepts. In this sense, it becomes possible and natural to see the boundary of Littlestone within VC classes, as this section will explain. 

To proceed more formally, first let us re-present the main learning concepts discussed so far, and some others from recent work in the literature, as statements about existence of approximations. This will allow us to formulate an equivalence theorem \ref{t:equivalence}, the main result of the section.  Theorem \ref{t:equivalence} compares several recent learning frameworks characterizing Littlestone classes using this common language, and suggests common themes which have so far not emerged from various pairwise comparisons. (The addition of PEC also helps clarify the picture.) 
At the end of the section we will return to the question of what underlies learnability of Littlestone classes, and suggest also an informal answer. 

\begin{defn}[Approximations] 
Say that the algorithm $\mca$ gives 
\[ \mbox{ $(\alpha, \beta, n)$-approximations for $\mch$ } \]if for any distribution $D$ which is realizable by $\mch$, 
with probability at least $1-\alpha$ over all samples of size $n$ drawn from $D^n$, $\mca(S)$ has $D$-population loss at most $\beta$.
% is $\beta$-close to $h$ w.r.t $D$.  
\end{defn}

Then the definition of PAC learnability of a class $\mch$ says precisely that it has $(\alpha, \beta, n)$-approximations for all $\alpha,\beta$ and sufficiently large $n=n(\alpha,\beta)$.

\textbf{Note}: The next few definitions extend the definition of PAC learning and so we assume a quantification over $\alpha, \beta, n$ and the distribution $D$.

\begin{defn}[Globally stable approximations]\label{def:globst} Given $\eta > 0$
say that $\mca$ gives $\eta$-globally stable frequent approximations for $\mch$ if in addition, 
one specific output hypothesis occurs with probability at least $\eta$.

Say that $\mch$ has globally stable approximations if there exists a fixed positive constant $\eta>0$ and a (possibly randomized) learning rule $\mca$ such that for every $\alpha,\beta>0$ there exists $n$
for which $\mca$ gives an $\eta$-globally stable $(n,\alpha,\beta)$-approximations for $\mch$.
\end{defn}

Note that the $\eta$-frequent hypothesis might depend on $D$ and on $n$, but it does not depend on the input sample drawn for $D$.
While we do not explicitly require that the $\eta$-frequent hypothesis has small loss, this is the intended interpretation (it is implied by setting the confidence parameter $\beta<\eta$).

\begin{defn}[Differentially private approximations]\label{def:privapx} Given $\eps,\delta > 0$
say that $\mca$ gives $(\eps,\delta)$-differentially private frequent approximations for $\mch$ if in addition, 
$\mca$ is $(\eps,\delta)$-differentially private.

Say that $\mch$ has differentially private approximations if there exist a fixed positive constant $\eps>0$, a negligible\footnote{Recall that a function $\delta(n):\mathbb{N}\to \mathbb{R}_{\geq 0}$ is negligible if $\lim_{n\to \infty}\delta(n)\cdot P(n) = 0$ for every polynomial $P$.} $\delta=\delta(n)$, and a (possibly randomized) learning rule~$\mca$ such that for every $\alpha,\beta$ there exists $n$ for which $\mca$ gives an $(\eps,\delta)$-differentially private $(n,\alpha,\beta)$-approximations for $\mch$.
\end{defn}

The next couple of definitions use basic quantities from information theory: the Kullback-Leibler (KL) divergence, and Shannon's mutual information (see e.g.~\cite{Cover2006}).

\begin{defn}[PAC Bayes stable approximations]\label{def:pbs}
Say that $\mca$ gives $r$-PAC Bayes stable frequent approximations for $\mch$ if in addition there exists a prior distribution $p=p(D)$ over the output space of $\mca$ such that 
	\[\Ex_{S}\Bigl[\mathtt{kl}\bigl(\mca(S) ~||~ p\bigr)\Bigr] \leq r,\]
	where $\mca(S)$ denotes the posterior distribution over the output space of $\mca$ given the input sample $S$.
	In other words, the posterior is similar (according to the KL-divergence) to a prior distribution which depends only on the distribution $D$ (but not on the data-sample $S$).

Say that $\mch$ has PAC Bayes stable approximations if there exist a sublinear $r=r(n)$ and a (possibly randomized) learning rule $\mca$ such that for every $\alpha,\beta$ there exists $n$ for which $\mca$ gives an $r$-PAC Bayes stable $(n,\alpha,\beta)$-approximations for $\mch$.
\end{defn}

\begin{defn}[Information stable approximations]\label{def:is}
Say that $\mca$ gives $r$-information stable frequent approximations for $\mch$ if 
	\[I\bigl(S,A(S)\bigr) \leq r,\]
	where $I(\cdot, \cdot)$ is Shannon's mutual information function.
	In other words, the output hypothesis does not give away more than $r$-bits of information on the data-sample~$S$.

Say that $\mch$ has information stable approximations if there exist a sublinear\footnote{Recall that a function $r(n):\mathbb{N}\to \mathbb{R}_{\geq 0}$ is sublinear if $\lim_{n\to \infty} r(n)\cdot n^{-1} = 0$.}~$r=r(n)$ and a (possibly randomized) learning rule $\mca$ such that for every $\alpha,\beta$ there exists~$n$ for which $\mca$ gives an $r$-information stable $(n,\alpha,\beta)$-approximations for $\mch$.
\end{defn}

The new learning framework of the present paper, PEC learning, has a place of honor in this list.
%But since its definition was already given with the present ideas in mind, it will be most convenient to refer to it directly.}

\begin{defn}[Eventually stable almost everywhere correct approximations]\label{def:pecs}
Say that $\mca$ has eventually stable eventually correct approximations for $\mch$ in $d$~steps 
if~$\mca$ is a PEC-learner for $\mch$ with at most $d$ mind-changes.
\end{defn}

Before proceeding to Theorem \ref{t:equivalence}, let us compare PEC learning to the other items discussed. 
The PEC notion of stability is distinguished in several ways: (i) the SOA algorithm satisfies it. (The SOA is also the basis of the other proofs but the resulting stable algorithms are rather convoluted and not so elegant.) (ii) it is asymptotic in the sense that the correct hypothesis is only outputted eventually, (iii)~the error of the eventually outputted hypothesis is zero (rather than at most~$\eps$).
The last point implies some sort of sample independence of the output: indeed, zero error means that two outputs which are trained on two independent random sequences agree almost everywhere.
Another difference is that PEC requires a bounded number of mind-changes which is stability of the learning process rather than of the output (or distribution over outputs) as in the other definitions.

In short, PEC learning is a simpler definition which is arguably more basic, though also less practical: it is a natural analog of PAC learning which captures learning with \emph{zero} (=measure zero) error, by asking for uniform finiteness only on the number of mind-changes, not the sample size. 

We can now state the equivalence: 

\br

%\newpage
\begin{theorem}[Littlestone classes characterized via existence of approximations] \label{t:equivalence}
For $(X, \mch)$ the following are equivalent, and equivalent to the statement that $\mch$ is a Littlestone class. 
\begin{enumerate}
\item {\bf Repetition.} 
    \begin{itemize}
        \item[(a)] The class $\mch$ has globally stable approximations in the sense of Definition~\ref{def:globst}.
        \emph{Slogan: a unique hypothesis is outputted frequently and quickly; frequency might be small and the error might be nonzero.}    
        \item[(b)] The class $\mch$ has eventually stable approximations in the sense of Definition~\ref{def:pecs}.
        \emph{Slogan: a unique (up to measure zero) hypothesis is outputted with only a few mind changes but not necessarily quickly; frequency is one  and error is zero.}
    \end{itemize}
\item {\bf {Concentration.}} 
    \begin{itemize}
        \item[(a)] The class $\mch$ has PAC-Bayes stable approximations in the sense of Definition~\ref{def:pbs}.
        \emph{Slogan: posteriors (the distributions over output hypotheses conditioned on the input sample) are typically close to a data-independent prior.}    
        \item[(b)] The class $\mch$ has information stable approximations in the sense of Definition~\ref{def:is}.
        \emph{Slogan: the mutual information between the output hypothesis and the input data is small; 
        informally, thought of as random variables, the output hypothesis is nearly independent from the sample.}
    \end{itemize}
\item {\bf {Local stability.}} 
    \begin{itemize}
    \item[(a)] The class $\mch$ has differentially-private approximations in the sense of Definition~\ref{def:privapx}.
    \emph{Slogan: small changes in the input barely change the output.}    
    \end{itemize}
\end{enumerate}
\end{theorem}

\begin{proof}
Throughout this proof, when we give references to prior work, we mean: the immediate translation of the referenced theorem to the language of approximations introduced here gives the desired result. 

The equivalence between Items 1(a), 3(a), and finite Littlestone dimension is proved in~\cite{ABLMM}.

The equivalence between Item 1(b) and finite Littlestone dimension is the content of Theorem~\ref{t:stablepec}.

The equivalence of Item 2(b) and finite Littlestone dimension appears in~\cite{PNG22}.
That Item 2(a) implies finite Littlestone dimension follows from~\cite{lm20}.
Thus, to complete the equivalence of Item 2(a) with finite Littlestone dimension,
it suffices to show that Item 2(b) implies Item 2(a). Indeed, this follows from the basic identity
\begin{align*}
 I\bigl(S,A(S)\bigr) = \Ex_{S}\mathtt{kl}\bigl(\mca(S) ~ || ~ \Ex_{S}[A(S)]\bigr). 
\end{align*}
(Recall that $A(S)$ is the posterior distribution over the output hypothesis space given the input sample $S$;
thus, $\Ex_{S}[A(S)]$ is the expected posterior, taken with respect to the input sample $S$.)
Thus, picking the prior as the expected posterior witnesses that $\mch$ satisfies Item 2(a).
\end{proof}

The next discussion proposes a moral for Theorem \ref{t:equivalence}.  

\begin{disc}[Key features of Littlestone learnability] \label{d:key}
The equivalent conditions of Theorem \ref{t:equivalence} provide different manifestations for the thesis that Littlestone classes have learning rules that are sample-independent in the sense that the output is essentially just a function of the population (but beyond this, it does not really depend in an additional way on the actual data it was trained on). We suggest as a common slogan: 
\[ \mbox{The output is distribution-dependent but data-independent.} \]
\end{disc}

\noindent 
Notice that in \ref{d:key} the stability refers to the posterior distribution over output hypotheses, and not just the output hypothesis itself.
Put differently, the output hypothesis is stable when seen as a random variable, given the input sample $S$. 
    That is, the relevant learning rule here is a randomized learner, and hence given an input sample $S$
    it induces a distribution over output hypotheses (which is instantiated once the randomness of the algorithm is specified). It is this distribution that is required to be stable: in differential stability we require local-stability (replacing one example in the input sample changes this distribution only a bit), 
    and in the PAC-Bayes and information stability we require that for a typical sample, this distribution is close to a data-independent distribution over hypotheses (which is sometimes called the prior).

A simple example may clarify this important distinction between 
distributions over outputs, and outputs. In plain language, 
an initial objection to Discussion \ref{d:key} might go as follows: 
Isn't PAC learnability itself a form of stability, in the sense that the 
output hypotheses are all likely to be within an $\epsilon$-ball of each other, 
and so likely to be equal at many points in the domain? If so, why do 
VC classes not fall under \ref{d:key}?  To see the difference \ref{d:key} is pointing 
out, consider the following.

\begin{expl} \label{e:thresholds} Consider the example where 
$X = \mathbb{Q}$ or $X = \mathbb{R}$ and $\mch$ is the class thresholds (i.e., cuts) over $X$. This is a VC class but not a Littlestone class.  
Consider a simple deterministic algorithm $\mca_0$ that PAC learns $\mch$
whose operation on a given finite sample $S$ of $m$ elements is as follows. 
Identify $L(S)$, the rightmost element in the sample labeled ``$0$,'' and $R(S)$, the 
leftmost element in the sample labeled ``$1$,'' and output the hypothesis corresponding to the threshold at $(L(S)+R(S))/2$.  
Clearly $\mca_0$ is sensitive to the specific sample, since different values of~$L(S)$ and 
$R(S)$ may lead to different outputs. However, as the initial objection points out, 
the fact that $A_0$ is a PAC learner means we still expect that the outputted hypotheses agree on many $x \in X$. 
\end{expl}

To see why \tcb{finite} Littlestone offers a greater stability guarantee, notice that since $\mca_0$ from 
Example \ref{e:thresholds} is deterministic, the distribution over outputs induced by it is a Dirac distribution: 
it is supported on the single output hypothesis $\mca_0(S)$ and gives it probability one. 
Now observe that if $\mca$ receives two samples $S, S^\prime$ in which $L(S) + R(S) \neq L(S^\prime) + R(S^\prime)$, then even though $\mca_0(S)$ and $\mca_0(S^\prime)$ may be close  as hypotheses, $\mca(S)$ and $\mca(S^\prime)$ will be completely different as distributions/random variables (since they are probability distributions which assign probability 1 to different things).  By contrast, Theorem \ref{t:equivalence} and Discussion \ref{d:key} are pointing out that in Littlestone classes, the stability achieved is at the level of the output distributions (and this is characteristic of \tcb{finite} Littlestone as the theorem describes).

The stability observed in VC classes (or PAC learnability) can be summarized as the assertion that when a particular behavior is learnable, any algorithm that learns it will produce hypotheses that exhibit similar behavior. On the other hand, the stability demonstrated in Littlestone classes (or private PAC learnability) takes it a step further by asserting that the learning algorithm itself, which maps training sets to generated hypotheses, is stable. 

\begin{disc}
The proof of \ref{t:equivalence} proceeds largely via the equivalence of each itemized condition to $\mch$ being a Littlestone class.  It could certainly be interesting to improve our understanding of the various arrows directly. 
\end{disc}

\br
\vspace{5mm}

\section{The Algorithmic Unstable Formula Theorem} \label{s:revisited} 

\setcounter{theoremcounter}{0}

In this section, we draw on 
recent work and on the above to lay out a full analogue of Shelah's Unstable Formula Theorem (see the Appendix) in the learning setting, 
Theorem \ref{t:unstable}, with algorithmic 
arguments taking the place of the infinite. 
Note that although it is of course equivalent, we state the conditions equivalent to 
stability, not instability as in the original theorem.  
It isn't necessary to already know the Unstable Formula Theorem to understand Theorem \ref{t:unstable}, but there is a resonance 
in both directions which is thought-provoking.  So we will first discuss some problems that had to be solved to give an algorithmic version, in order to emphasize what may be surprising. 

Although this section may seem natural in hindsight, as the paper's title suggests, we feel it is a central creative contribution of the paper. It took us a long time to see that it could be done, and done in a way we found convincing. 

The items in the unstable formula theorem may 
for present purposes be grouped into four kinds. In model theoretic language, 
there are conditions on counting types, combinatorial conditions on formulas, statements about 
bounds on rank, and about definability of types (we shall explain the algorithmic analogues in due course). 

We begin with the first group of items concerning conditions on counting types, 
explain their infinite nature, and the challenge of expressing finite, combinatorial analogues. (It also has a place of honor in our own work: we first made the connection of counting types to algorithms in \cite{agnostic} and it is what led us to wonder if an Algorithmic Unstable Formula theorem might exist.) 
In combinatorial language, the Sauer-Shelah-Perles (SSP) lemma asserts that given a class $\mch$ with VC dimension $d < \infty$ 
and a finite set $\{ x_1, \dots, x_n \} \subseteq X$, the number of intersection patterns $h \cap \{ x_1, \dots, x_n \}$ realized by 
sets $h \in \mch$ is polynomially bounded by $\binom{n}{\leq d}$. It turns out that this useful inequality lacks the expressivity to 
detect finite Littlestone dimension. For example, consider the families 
\[ \mch_1 = \bigl\{ \{ x \} : x \in \mathbb{Q} \bigr\} \mbox{ and } \mch_2 = \bigl\{ \{ x \in \mathbb{Q} : x \leq t \} : t \in \mathbb{R} \bigr\}. \]
These are two families of sets of rational numbers with VC dimension one that satisfy the SSP lemma with equality, that is, 
for every finite $\{ x_1, \dots, x_n \}$ both families realize exactly $\binom{n}{\leq 1} = n+1$ intersection patterns. However 
$\mch_1$ is Littlestone whereas $\mch_2$ is not. 

In model theoretic language, traditionally, it was not possible to ``see'' stability by counting types over finite sets 
and this created a certain blindness when doing stability in the finite. In model-theoretic language, letting $\vp$ be the graph edge relation, compare the number of 
$\vp$-types over a set of size $n$ (as $n$ grows) in an infinite matching (corresponding to $\mch_1$), versus an infinite half-graph (corresponding to~$\mch_2$). Since any finite linear order is discrete, these two counts differ by at most one, although the half-graph is unstable whereas the matching is not. 

However, in the model theoretic case there is an additional insight, namely, the unstable formula theorem says that when counting $\vp$-types over infinite sets, quantifying over all models of the theory, there is a characteristic difference.  Indeed, returning to the combinatorial example, consider an 
infinite rational interval~$I_{a,b} = [a,b]_{\mathbb{Q}} = \{ x \in \mathbb{Q} : a \leq x \leq b \}$, with $a<b$. Observe that  
$\mch_1$ realizes only a countable number of intersection patterns over $I_{a,b}$ (corresponding to all \emph{rational} numbers between 
$a$ and $b$), whereas $\mch_2$ realizes continuum many intersection patterns (corresponding to all \emph{real} numbers between 
$a$ and $b$). Such infinite counting arguments are used to characterize stability in model theory. 

It turns out that it is possible to characterize Littlestone classes by counting \emph{algorithms} over finite sets.  The following may be 
called an online, dynamic, or adversarial SSP lemma.  

\begin{defn} \label{d:ossp}
Say that $\mch$ satisfies the \emph{online Sauer-Shelah-Perles lemma}
when for every $n \in \mathbb{N}$ there exists  a collection of $\binom{n}{\leq d}$ algorithms or dynamic sets $\mca:X^\star\to\{0,1\}$ 
that cover all the $\mch$-realizable sequences of length $n$; that is, for every sequence $\{(x_i,y_i)\}_{i=1}^n$ that is realizable by $\mch$
there is $j\leq \binom{n}{\leq d}$ such that $A_j(x_1,\ldots, x_i) = y_i$.
\end{defn}
Shalev-Shwartz, Pal, and Ben-David proved that every Littlestone class satisfies  Definition~\ref{d:ossp} with $d$ being the Littlestone dimension (see Lemma 12 in~\cite{bdps}); the converse direction that classes with unbounded Littlestone dimension do not satisfy Definition~\ref{d:ossp} also follows implicitly by~\cite{bdps}, and is noted explicitly e.g.\ in Lemma~9.3 in~\cite{adversarial}. 
The understanding that lemmas of this kind can be seen as  
analogues of counting types has its origins in \cite{agnostic}. \tcb{As explained in the discussion 
around \cite{agnostic} Theorem 5.1, the functions in Definition \ref{d:ossp} are essentially SOAs which make at most $d$ mistakes.}

\br
Next we turn to the problem of sampling.  
The most important property of VC classes is arguably uniform convergence (a.k.a $\eps$-approximations). This property asserts that any VC class $\mch$ satisfies a uniform law of large numbers in the sense that a sufficiently large i.i.d random sample from any distribution over the domain $X$ is likely to represent the measures of all sets in $\mch$ simultaneously and uniformly. From a technical perspective uniform convergence is naturally paired with the Sauer-Shelah-Perles counting lemma. In fact, the optimal rate of uniform convergence is quantitatively characterized by the minimal polynomial degree for which the SSP lemma is satisfied.  
It is therefore very natural to ask whether the ability to detect Littlestone classes via an adapted counting lemma also gives rise to 
an ability to detect Littlestone classes via an adapted law of large numbers.

\begin{defn} \label{d:olln} 
Say that $\mch$ has online $\epsilon$-approximations, or that $\mch$ satisfies the Adversarial Law of Large Numbers, when the following holds. Consider a sequence $S=x_1,\ldots, x_N$  which is presented sequentially to a sampler that retains a uniformly random subsequence of length $n\leq N$. Then the following is satisfied even if the sequence is adaptively produced by an adversary which sees, after presenting each element $x_i$, whether $x_i$ was retained. With probability at least $1-\delta$, every $h\in \mch$ satisfies
\[\bigl\lvert\mu_S(h) - \mu_{\hat S}(h)\bigr\rvert\leq 
\eps_{\mch}(\delta,n)%\longrightarrow_{n\to\infty} 0 
%O\Bigl(\sqrt{\frac{\log(1/\delta)}{n}}\Bigr)
,\]
where $\mu_S(h) = \frac{1}{N}\sum_{i=1}^N 1[x_i\in h]$, and $\mu_{\hat S}(h) = \frac{1}{n}\sum_{j=1}^n 1[x_{{i_j}}\in h]$.
Above, $\eps_{\mch}(\delta,n)$ is a sequence satisfying $\lim_{n\to\infty}\eps_{\mch}(\delta,n) = 0$ for every fixed $\delta>0$. 
% the big oh notation conceals a positive constant which depends only on the class $\mch$.\footnote{A more detailed and asymptotically tight bound is $O(\sqrt{\frac{d + \log(1/\delta)}{n})}$, where $d$ is the Littlestone dimension, and this time the big oh notation conceals a universal numerical constant.}
\end{defn}
It was recently shown by \cite{adversarial} that Definition~\ref{d:olln} is satisfied by $\mch$ exactly when it has a finite Littlestone dimension $d$, in which case the optimal rate $\eps_\mch$ satisfies 
\[\eps_\mch(\delta,n)=O\Bigl(\sqrt{\frac{d + \log(1/\delta)}{n}\Bigr)},\]
which is the same function of $d,\delta$ %like
as in the classical uniform convergence theorem for VC classes. 

In both Definitions \ref{d:ossp} and \ref{d:olln}, the difference between VC classes and Littlestone classes arises in the online or adversarial aspect: the adversary is able to change the sequence of examples given to the sampler in real time in the Littlestone case, compared to the presentation of an iid sample in the VC case.  Although this language seems far from model theoretic, %observe that  
in comparing the changes in combinatorial structure they detect, the adversarial aspect in the finite has certain parallels to quantification over all possibilities 
in the monster model in model theory. 

Next we turn to combinatorial properties.  These correspond to items (3) and (4) 
in the unstable formula theorem, from which it follows that one is finite if and only if the other is finite. The first direct bounds in the finite case, 
which remain best to date, are due to Hodges, see \cite{hodges} 6.7.8-9: in our language, if $\mch$ has $k$-half graphs then it has $d$-trees for $d$ about $\log(k)$, and if $\mch$ has $d$-trees it has $k$-half graphs for $k$ about $\log(d)$.\footnote{A proof in the language of concept classes was included in the appendix of \cite{almm}.} 

These two items were, to our knowledge, the first items in the unstable formula theorem to be applied to understand finite graphs, in \cite{MiSh:978}. 

\begin{defn} \label{d:hg}
$\mch$ has finite half-graphs if for some $k < \omega$  
there do not exist distinct $x_1, \dots, x_k \in X$, $h_1, \dots, h_k \in \mch$ such that $x_i \in h_j$ if and only if $i < j$.\footnote{This is sometimes phrased as  ``$\mch$ has finite Threshold dimension.''} We may say $\mch$ is $k$-stable to emphasize the value of $k$. 
\end{defn}

\begin{defn} \label{d:tree}
Say that $\mch$ has finite trees if for some $t < \omega$ 
there do not exist distinct $\langle a_\eta : \eta \in {^ {t\geq} 2} \rangle$ from $X$ and distinct
$\langle h_\rho : \rho \in {^{t+1} 2} \rangle$ from $\mch$, so that for every such $h_\rho$ and $a_\eta$, if 
$\eta^\smallfrown \langle 0 \rangle \tlf \rho$ then $a_\eta \notin h_\rho$ and if $\eta^\smallfrown \langle 1 \rangle \tlf \rho$ then $a_\eta \in h_\rho$. 
\end{defn}

Definition \ref{d:tree} is %equivalent to
{one more than} a bound on the height of a mistake-tree (from \S \ref{s:littlestone} above) and we will use the two versions interchangeably, and refer to the minimal 
such $t$-1, if finite, as $\Ldim(\mch)$.  It is useful that these items connect an evidently self-dual property 
\ref{d:hg} with a notion of dimension \ref{d:tree} which is not obviously self-dual (see discussion in \cite{agnostic}). 

Next we turn to notions of rank.  In light of \cite{MiSh:978}, the relevant principle in the algorithmic 
case seems to be the derived notion of majority.  We mentioned that the two combinatorial properties above were used to understand the 
structure of finite stable graphs in \cite{MiSh:978}, towards proving the stable regularity lemma and stable Ramsey theorems. 
In that paper, the theory of stability was in some sense ``miniaturized'' to find aspects analogous to stability theory in the finite. 
The following, stated there for graphs, was a key definition (for more on its model theoretic antecedents and on  
stable regularity, see \cite{MiSh:E98}). 

\begin{defn} 
Say that the finite set $A \subseteq X$ is $\epsilon$-good if 
for any $h \in \mch$, either 
$\{ a \in A : h(a) = 0 \}$ or $\{ a \in A : h(a) = 1 \}$ has size $< \epsilon |A|$. 
Say that $\mch$ has linear-sized $\epsilon$-good sets if there is $c = c(\epsilon)$ such that any finite 
$Y \subseteq X$ has a subset $A$ of size $\geq |Y|^c$ which is $\epsilon$-good. 
\end{defn}

There is an important extension of this definition which is naturally stated for graphs, and for concept classes only with slightly more 
work, so we do not include it in the equivalence below but record it here. Let $a\sim x$ mean an edge exists. 
In a finite graph $G$, given $A \subseteq G$, if $|\{ a \in A : x\sim a \}| < \epsilon|A|$, write \tcb{ $\trv(x,A) = 0$}, 
and if $|\{ a \in A : x \not\sim a\}| < \epsilon|A|$, write \tcb{ $\trv(x,A) = 1$}, for the 
truth value or majority opinion. The following definition, also from \cite{MiSh:978}, says that most 
elements in $B$ have the same majority opinion when faced with any good set. 

\begin{defn} In a graph $G$, 
say the finite set $B \subseteq G$ is $\epsilon$-excellent if for any finite $A \subseteq G$ which is $\epsilon$-good, there is 
$\trv = \trv(B,A) \in \{ 0 , 1 \}$ so that for all but $\epsilon|B|$ elements $b \in B$, we have that 
$\trv(b, A) = \trv$.
\end{defn}

The statement that any $k$-stable graph has linear-sized $\epsilon$-good sets for any $\epsilon$, and indeed linear-sized 
$\epsilon$-excellent sets for any $\epsilon < \frac{1}{2^{2^k}}$ or so, was proved in 
\cite{MiSh:978}. 
We note that good and excellent sets have since played key roles in lines of work related to versions of 
stability in arithmetic regularity \cite{TW} and graph limits \cite{CM}. 
Existence of linear-sized $\epsilon$-excellent sets for any $\epsilon < \frac{1}{2}$ under stability was proved in \cite{agnostic}. 

Although we have used the word majority, there are at least two distinct candidates for a notion of majority in Littlestone classes: 
the majority arising from counting measure, as in $\epsilon$-good, and the majority arising from Littlestone dimension.  
In \cite{agnostic} it was shown that these two notions of majority in some sense densely often agree, and that Littlestone classes are characterized by admitting a simple axiomatic notion of majority in the sense of \cite{agnostic} Definition 7.7; for reasons of length, we just give a pointer:

\begin{defn} \label{d:large} 
Say that $\mch$ admits an axiomatic notion of largeness if  it satisfies~\cite{agnostic} Definition $7.7$. 
\end{defn}

Finally, we turn to definability. This was the subject of Section \ref{s:frequent} above, 
but we can add one last piece of the puzzle. 
Conditions (8)-(9) in the unstable formula theorem deal with definability of types: the difference is that (8) is about existence 
of definitions, whereas (9) is about existence of definitions which are in some sense canonical or uniform. By analogy, let us 
propose that in Theorem \ref{t:equivalence}, PEC learning plays a role similar to (9) because of its use of the SOA. Indeed, a 
reader familiar with both the definition via $R(x=x,\vp, 2)$-rank which may be used to prove (9) 
and with the SOA will see the analogy is almost exact.

\br

We now arrive to the statement of the theorem. 

\begin{theorem}[The ``Algorithmic'' Unstable Formula Theorem] \label{t:unstable} For $(X, \mch)$ the following are equivalent and equivalent to the statement that 
$\mch$ is a Littlestone class.  
\end{theorem}

\noindent \textbf{A. Counting and sampling} 

\begin{enumerate}
\item[A1.] 
{\it $\mch$ satisfies the online Sauer-Shelah-Perles lemma in the sense of \ref{d:ossp}.}

\item[A2.] {\it $\mch$ has online $\epsilon$-approximations %and adversarial laws of large numbers 
in the sense of Definition \ref{d:olln}.}  

\item[A3.] %{\bf Online learnability.}
{\it $\mch$ is online learnable.} 

\end{enumerate}

\noindent \textbf{B. Combinatorial parameters}

\begin{enumerate}
\item[B1.] {\it $(X, \mch)$ has finite half-graphs.} 

\item[B2.] {\it $(X, \mch)$ has finite Littlestone dimension.} 

\item[B3.] {\it The dual class $(\mch, X)$ has finite Littlestone dimension.}
\end{enumerate}

\noindent \textbf{C. Majorities} 
\begin{enumerate}
\item[C1.] {\it Existence of linear-sized $\epsilon$-good subsets.}

\item[C2.] {\it $\mch$ admits an axiomatic notion of largeness.} \end{enumerate}

\noindent \textbf{D. Frequent approximations}

\begin{enumerate}
\item[D1.] {\it $\mch$ satisfies the equivalent conditions of Theorem 
\ref{t:equivalence} above: repetition, concentration, local stability, and indeed, admitting eventually stable approximations by a \emph{canonical} algorithm, the SOA.}

\end{enumerate}

\begin{proof}  
The equivalence of (A1) and finite Littlestone dimension is established in~\cite{adversarial,bdps}: \cite{bdps} explicitly show that Littlestone classes satisfy (A1), but the converse is implicit in their work. The equivalence is stated in full in~\cite{adversarial}, Lemma 9.3.
The equivalence of (A2) and finite Littlestone dimension is the main result of~\cite{adversarial}, and the equivalence of (A3) with finite Littlestone is the main result of~\cite{bdps}.

The equivalence of (B1) and (B2) translates so-called Hodges lemma, which gave a finite version of a 
correspondence between orders and trees in the unstable formula theorem ((3) iff (4) on page \pageref{ct-image} below).  A statement and 
proof of (B1) iff (B2) in combinatorial language can be found in \cite{hodges} 6.7.8-9 and a statement and proof in the present language is included in the appendix to \cite{almm}.  The equivalence of (B3) for the dual class follows from the fact that half-graphs are evidently self-dual so 
$(X, \mch)$ has finite Threshold dimension if and only if $(\mch, X)$ has finite Threshold dimension. 
Of course, the dimension may be different. 

(B2) is equal to Littlestone by definition. 

The existence of $\epsilon$-good sets in $k$-stable graphs is due to \cite{MiSh:978}. 
The equivalence of item (C1) and Littlestone was recorded in our language in \cite{agnostic} Claim 2.4. 
The equivalence of item (C2) and Littlestone is 
\cite{agnostic} Theorem 7.11.

(D1) if and only if Littlestone is the content of Theorem \ref{t:equivalence} above, plus the clause about the SOA from Theorem \ref{t:stablepec}.
\end{proof}

\vspace{10mm}

\section{Some notes for future work}

When things are new the path is still being made. 
But as a gesture of appreciation to the reader who has understood so far, let us list several informal questions.   

(1) In Theorem 5.11,  one can consider replacing the assumption ``$\vp$ is a stable formula'' by ``$D$ is a distribution under which $\vp$ is a stable formula'' suitably defined. 

(2) Many of the arrows in the theorems of equivalences we have given above are not direct, and indeed, direct proofs are not currently obvious and might be worth while. 
Say, could the adversarial items in Theorem 7.8(A) and the concentration phenomena in Theorem 6.7  be brought closer together? 

(3) We expect that future work will add further equivalences to some of these theorems, which may enrich the picture. 

(4) Are there better or finer notions of rank visible in the finite which more clearly reflect the variety of model theoretic ranks visible in the infinite? 

(5) Are there classes of algorithms which are, broadly, neither adversarial as in online, nor statistical as in PAC, which can be marshalled to give interesting items in the algorithmic Unstable Formula Theorem?

\newpage

\section*{Appendix: The unstable formula theorem}

Here is the statement of Shelah's unstable formula theorem as it appears in \cite{Sh:a}, Chapter II, Theorem 2.2. 

\vspace{5mm}
\begin{center} \label{ct-image}
\includegraphics[width=110mm]{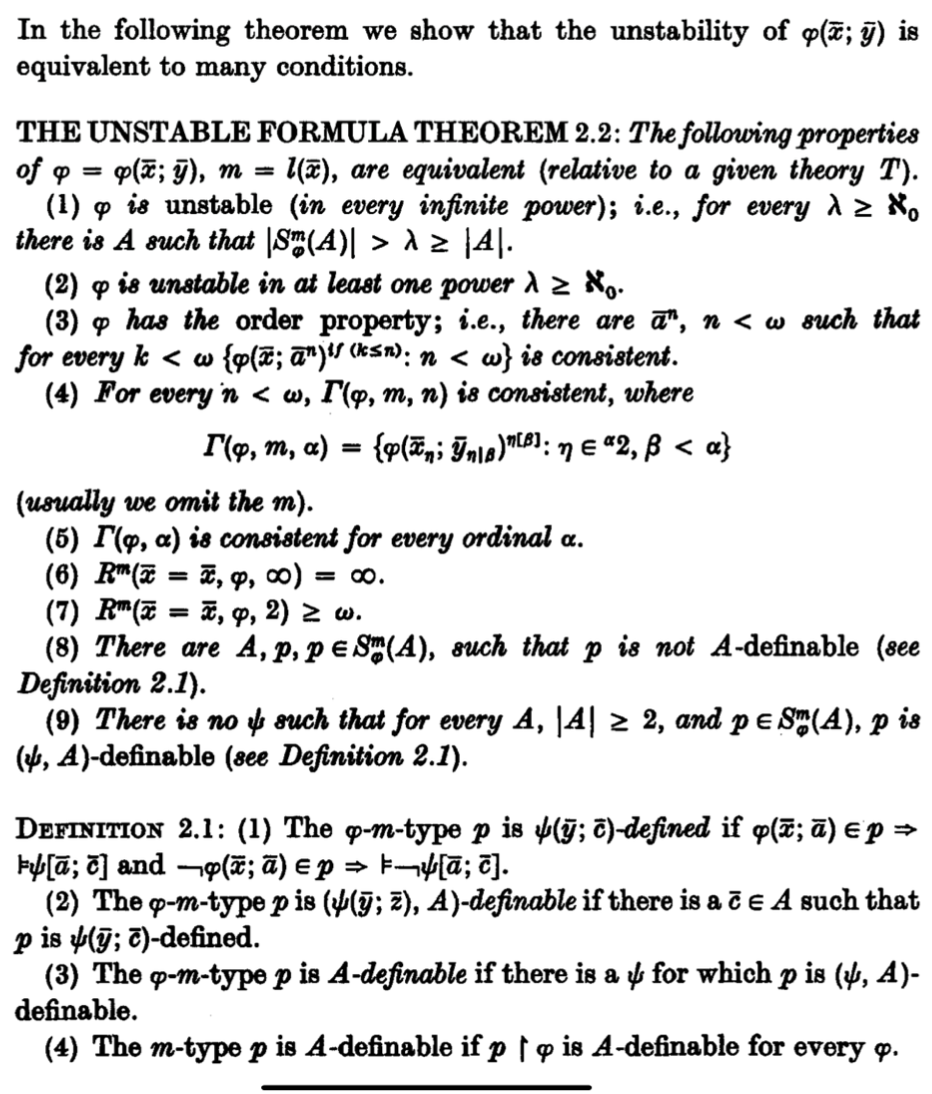}

\end{center}

\newpage

%\vspace{10mm}
\subsection*{Acknowledgments} 
Many people encouraged us after seeing the first \tcb{(2022)} version of this manuscript.  Thank you. 

\tcb{We thank the anonymous referee for a thoughtful reading.} 

MM: Research partially supported by NSF-BSF 2051825. 
Some expository aspects of this work benefitted from presentation at the ``Combinatorics Meets Model Theory'' conference at Cambridge and in graduate course lectures 
at Chicago.

SM: Shay Moran is a Robert J.\ Shillman Fellow; he acknowledges support by ISF grant 1225/20, by BSF grant 2018385, by an Azrieli Faculty Fellowship, by Israel PBC-VATAT, by the Technion Center for Machine Learning and Intelligent Systems (MLIS), and by the the European Union (ERC, GENERALIZATION, 101039692). Views and opinions expressed are however those of the author(s) only and do not necessarily reflect those of the European Union or the European Research Council Executive Agency. Neither the European Union nor the granting authority can be held responsible for them.

\bibliographystyle{amsplain}

\begin{thebibliography}{00}


\bibitem{adversarial}
N. Alon, O. Ben-Eliezer, Y. Dagan, S. Moran, M. Naor, E. Yogev. 
``Adversarial laws of large numbers and optimal regret in online classification.'' STOC 2021. 

\bibitem{almm} N. Alon, R. Livni, M. Malliaris, S. Moran. ``Private PAC learning implies finite Littlestone dimension.'' STOC 2019. 

\bibitem{ABLMM} N. Alon, M. Bun, R. Livni, M. Malliaris, S. Moran. ``Private and online learning are equivalent.'' 
J. ACM (2022) 69, 4 pp. 1--34. 
%J. ACM, 2022. 

\bibitem{BLM} M. Bun, R. Livni, S. Moran. ``An equivalence between private classification and online prediction.'' FOCS '20.

\bibitem{bdps} S. Ben-David, D. P\'{a}l, Shalev-Shwarz. ``Agnostic Online Learning.'' COLT '09.

\bibitem{bdss} Shai Ben-David and Shai Shalev-Schwarz. \emph{Understanding Machine Learning: From Theory to Algorithms.} 
Cambridge University Press, 2014. 

\bibitem{universalearning2020} O. Bousquet, S. Hanneke, 
S. Moran, R. van Handel, A. Yehudayoff. ``A Theory of Universal Learning.'' STOC 2021. 

\bibitem{Cover2006} T.M Cover and J.A Thomas.
\emph{Elements of Information Theory 2nd Edition.}
Wiley Series in Telecommunications and Signal Processing 2006.


\bibitem{chase-freitag} H. Chase and J. Freitag. ``Model theory and machine learning.'' Bulletin of Symbolic Logic, 25(03):319--332, 2019. 

\bibitem{CM} L. N. Coregliano and M. Malliaris. ``Countable Ramsey.'' ArXiv:2203.10396 (2022), 74 pages. 

\bibitem{dmns} C. Dwork, F. McSherry, K. Nissim, A. Smith. ``Calibrating noise to sensitivity in private data analysis.'' 
In \emph{Proceedings of the 3rd Conference on Theory of Cryptography}, TCC'06, 265--284, Berlin, Heidelberg, 2006. Springer. 

\bibitem{GGKM21} B. Ghazi and N. Golowich and R. Kumar and P, Manurnagsi. ``Sample-Efficient Proper PAC Learning with Approximate Differential Privacy.'' STOC 2021, 183--196. 

\bibitem{hodges} W. Hodges, \emph{Model Theory}. Encyclopedia of Mathematics and its
 Applications, vol. 42. Cambridge University Press, 1993. 

\bibitem{hrushovski} E. Hrushovski. ``Stable group theory and approximate subgroups.'' J. AMS 25 (2012) no. 1, 189--243. 


\bibitem{km} M. Karpinski and A. Macintyre. ``Polynomial bounds for VC dimension of sigmoidal and general Pfaffian neural networks.'' J. Computer System Sci (1997) 54(1):169--176.

\bibitem{laskowski} M. C. Laskowski. ``Vapnik-Chervonenkis classes of definable sets.'' J. London Math Soc (1992) 2(2):277--284. 

\bibitem{littlestone} N. Littlestone. ``Learning Quickly When Irrelevant Attributes Abound: {A} New Linear-threshold
               Algorithm.'' In \emph{Mach. Learn.} (1987) pps. 285--318. 

\bibitem{lm20} R. Livni and S. Moran. ``A Limitation of the PAC-Bayes Framework.'' In \emph{Annual Conference
               on Neural Information Processing Systems 2020} , NeurIPS 2020, December
               6-12, 2020, virtual. 


\bibitem{ls} R. Livni and P. Simon. ``Honest compressions and their application to compression schemes.'' In \emph{Conference 
on Learning Theory} (2013) pps. 77--92. 

\bibitem{mm-thesis} M. Malliaris, ``Persistence and regularity in unstable model theory.'' Ph.D. thesis, University of California, 
Berkeley, 2009. 

\bibitem{agnostic} M. Malliaris and S. Moran, ``Model theory and agnostic online learning via excellent sets.'' (2021) Submitted.  
ArXiv:2108.05569. 

\bibitem{MiSh:978} M. Malliaris and S. Shelah, ``Regularity lemmas for stable graphs.''  Trans. AMS 366 (2014), 1551--1585. (arXiv:1102.3904)

\bibitem{MiSh:E98} M. Malliaris and S. Shelah. ``Notes on the stable regularity lemma.'' Bull. Symb. Log. 27 (2021), no. 4, 415–425. 

\bibitem{PNG22} A. Pradeep, I. Nachum, M. Gastpar ``Finite Littlestone Dimension Implies Finite Information Complexity.'' {IEEE} International Symposium on Information Theory, {ISIT} 2022,
               Espoo, Finland, June 26 - July 1, 2022.

\bibitem{Sh:a} S. Shelah, \emph{Classification Theory}, North-Holland, 1978. Revised edition, 1990.

\bibitem{TW} C. Terry and J. Wolf, ``Stable arithmetic regularity in the finite field model.'' Bull London Math Soc 51 (2019) 70--88. 

\bibitem{vadhan} S. Vadhan, `` The complexity of differential privacy.'' In Tutorials on the Foundations of Cryptography, 
347--450. Springer, Yehuda Lindell, ed., 2017. 

\end{thebibliography}

\end{document}